\documentclass[conference]{IEEEtran}

\usepackage{cite}
\usepackage{amssymb}
\usepackage{amsgen}
\usepackage{amsfonts}
\usepackage{amsbsy}
\usepackage{amsmath}
\usepackage{amsthm}
\usepackage{color}
\usepackage{subfigure}
\usepackage{tikz}
\usepackage{ifthen}
\usepackage{url}

\renewcommand{\cal}{\mathcal}
\newcommand{\ra}{\rightarrow}

\newcommand{\Exp}[1]{E\left[#1\right]} 

\newcommand{\ceil}[1]{\left\lceil #1 \right\rceil}

\newcommand{\N}{\mathbb{N}}
\newcommand{\Z}{\mathbb{Z}}
\newcommand{\C}{\mathcal{C}}
\newcommand{\K}{\mathcal{K}}
\DeclareMathOperator*{\argmax}{arg\,max}

\newcommand{\hide}[1]{}

\newtheorem{theorem}{Theorem}
\newtheorem{assumption}[theorem]{Assumption}
\newtheorem{corollary}[theorem]{Corollary}
\newtheorem{remark}[theorem]{Remark}

\newtheorem{lemma}{Lemma}

\def\bb0{{\mathbb{0}}}

\def\bb{{\mathbf{b}}}

\def\b0{{\mathbf{0}}}
\def\opt{\mathsf{OPT}}

\def\bV{{\mathbf{V}}}

\def\b1{{\mathbf{1}}}

\def\bbE{{\mathbb{E}}}

\def\bbP{{\mathbb{P}}}

\def\cA{\mathcal{A}}

\def\cI{\mathcal{I}}

\def\cM{\mathcal{M}}

\def\sfi{{\mathsf{i}}}

\def\sf0{{\mathsf{0}}}

\def\nn{\nonumber}

\newboolean{wiopt}
\setboolean{wiopt}{false}
\newcommand{\conftext}[1]{\ifthenelse{\boolean{wiopt}}{{#1 }}{}}
\newcommand{\TRtext}[1]{\ifthenelse{\boolean{wiopt}}{}{#1}}

\begin{document}

\title{Non-asymptotic near optimal algorithms for two sided matchings}

\author{\IEEEauthorblockN{Rahul Vaze}
\IEEEauthorblockA{\textit{School of Technology and Computer Science} \\
\textit{Tata Institute of Fundamental Research}}
\and
\IEEEauthorblockN{Jayakrishnan Nair}
\IEEEauthorblockA{\textit{Department of Electrical Engineering} \\
\textit{IIT Bombay}}
}

\maketitle

\begin{abstract}
  A two-sided matching system is considered, where servers are assumed
  to arrive at a fixed rate, while the arrival rate of customers is
  modulated via a price-control mechanism. We analyse a loss model,
  wherein customers who are not served immediately upon arrival get
  blocked, as well as a queueing model, wherein customers wait in a
  queue until they receive service. The objective is to maximize the
  platform profit generated from matching servers and customers,
  subject to quality of service constraints, such as the expected wait
  time of servers in the loss system model, and the stability of the
  customer queue in the queuing model. For the loss system, subject to
  a certain relaxation, we show that the optimal policy has a
  \emph{bang-bang} structure.
  We also derive approximation guarantees for simple pricing policies.
  For the queueing system, we propose a simple bi-modal matching
  strategy and show that it achieves near optimal profit.
\end{abstract}

\section{Introduction}

Two-sided queues, where customers and servers both arrive to a
platform/aggregator and then wait to be matched, have been made fairly
popular by ride hailing applications like Uber and Lyft that match
passengers with drivers, meal delivery couriers like Grubhub and
DoorDash that match diners with delivery couriers, and crowdsourcing
platforms like Amazon MTurk, where tasks are matched to
workers/volunteers.  The objective of the platform is to maximize
profit, while improving the market efficiency.

To earn revenue, the platform sets a two sided or one sided price.
With two sided pricing, both the customers and the servers are
advertised a (possibly different) price, and customers willing to pay
the quoted price and servers who are willing to serve for the quoted
price enter the system. The platform profit is then the difference
between the two prices.  In a one sided price model, only the
customers are quoted a price, and customers who are willing to pay
enter the system. The server arrival rate is assumed to be fixed, and
insensitive to price. In this model, the platform keeps a fraction of
the revenue (made from customers) to itself and distributes the rest
uniformly across all servers.  In many practical systems, server
payoffs are better modelled as long-term rewards
\cite{banerjee2015pricing}, and thus considering one sided pricing is
reasonable.  In both models, the price is dynamically adjusted to
maximize profit, and to efficiently match supply and demand.

With two-sided queues, we consider two well studied models: the loss
model, and the queueing model.  For both models, we assume one sided
pricing, where the platform advertises a state dependent price that
determines the rate of arrival of customers, while the rate of arrival
of servers is fixed and insensitive to the platform price.

Under the loss model, servers arrive into a queue and wait, while
arriving customers are instantaneously matched to the head of the line
server, if any. A customer arriving when no servers are present is
lost. The objective we consider with the loss model is to maximize a
linear combination of the platform's profit and the expected delay
experienced by servers, and the goal is to find the optimal dynamic
price to maximize this objective function.

With the queuing model, both the servers and the customers wait in
their respective queues, and are matched appropriately. In this case,
we consider the objective of maximizing the platform's profit subject
to the stability of the customer queue. The goal is to design a
dynamic pricing and matching strategy to achieve optimal platform
profit while ensuring the stability of the customer queue.


\noindent {\bf Prior Work:} With a single sided queue (where server is
fixed and only customers arrive), \cite{paschalidis2000congestion}
considered a loss-system with finite resource capacity. A customer
arriving when all resources are occupied is lost, and the problem is
to decide the price of admission given the state of the system at each
time to maximize the expected payoff. Similar pricing models have been
considered for a queuing system, where customers wait in a queue, and
server's service distribution is exponential with a tuneable
parameter. Optimal admission control, with and without pricing, and
service parameter selection, so as to maximize a linear combination of
expected payoff and expected waiting time has been considered in
\cite{low1974optimal, stidham1989monotonic, george2001dynamic,
  ata2006dynamic, adusumilli2010dynamic, kim2018value}, assuming
Poisson arrivals. In particular, \cite{ata2006dynamic} exactly
characterizes the optimal policy, though not in closed-form, while
\cite{kim2018value} derives an asymptotically optimal dynamic pricing
policy.

Two sided queues have been considered in \cite{banerjee2015pricing,
  nguyen2018queueing, mahavir2020dynamic, sood2018pricing,
  kanoria2019near, yan2020dynamic, besbes2020surge, hu2020surge},
where both customers and servers arrive over time, and wait to be
matched. In particular, most of these papers are motivated from
ride-hailing applications such as Uber, Lyft, or meal delivery
services Grubhub, DoorDash, etc.  Some of the results
\cite{yan2020dynamic, banerjee2015pricing,
  besbes2020surge,hu2020surge} in this area focus on the importance of
dynamic pricing over static pricing. Under a suitable fluid scaled
limit of this system, \cite{banerjee2015pricing} showed that static
pricing is sufficient to optimize the objective function, and dynamic
pricing only helps in improving the robustness of the system.

An extension of this two-sided queue model with multiple types of
servers and customers has also been considered in
\cite{caldentey2009fcfs, adan2012exact, gurvich2015dynamic,
  hu2020dynamic, mahavir2020dynamic}, where an additional bipartite
matching decision has to be made. In \cite{caldentey2009fcfs},
limiting results of matching rates between certain customer and server
types with FCFS scheduling have been analyzed.  Without pricing, the
optimal matching to minimize the queueing cost has been analyzed in
\cite{gurvich2015dynamic} under a suitably scaled large system
limit. With the objective of minimizing the discounted reward obtained
by matching customers and servers over a finite horizon, while
accounting for the waiting costs, \cite{hu2020dynamic} analyzed fixed
pricing strategies, while dynamic pricing strategies were considered
in \cite{chen2020pricing, mahavir2020dynamic}. Most of these analyses
are presented using a large system scaling regime.

To summarize, in prior work, either exactly optimal but non-closed
form policies have been characterized, or structural properties of the
optimal policy have been established, or performance analysis of
simple strategies has been performed. Importantly, to the best of our
knowledge, any analysis of an explicit policy has been via an
asymptotic (large system) scaling.


\noindent {\bf Our Contributions:} Compared to prior work, the focus
of this work is to consider simple (one-sided) pricing strategies for
two-sided queues without considering any large system limits, e.g.,
fluid limits or asymptotics, with near-optimal
performance. Specifically, we model server arrivals as an exogenous
random process (independent of the system state and the
matching/pricing policy). On the other hand, customer arrivals are
modulated by the price~$p$ posted by the platform. Formally, we model
the customer arrival rate $\mu(p)$ as a non-increasing concave
function of the platform price $p$. We consider two variants of our
model: a loss model where customers not served upon arrival are
blocked, and a queueing model, where customers wait in a FCFS queue
until they are matched with a server. Our contributions for these two
models are as follows.

\subsubsection{Loss System}

For the loss system, we consider that the customer and server arrival
processes are Poisson, and the objective is to maximize a linear
combination of the platform's profit and the expected delay
experienced by servers.

When function $\mu(p)$ is linear in $p$ (which is an important case
\cite{lobo2003pricing} in pricing/profit models), we consider a
suitable relaxation of the objective, for which we show that the
optimal pricing strategy has a simple bang-bang style threshold
structure. This means that the price is set to be the maximum feasible
value when the number of servers is less than a threshold, and to the
minimum feasible value when the number of servers exceeds the same
threshold. This bang-bang structure of the optimal policy implies that
computing the optimal policy boils down to a simple one-dimensional
search. Moreover, the same structure can also be exploited to speed up
online learning of the optimal policy via reinforcement learning
(see~\cite{Roy2019}). While the relaxation we consider is quite
similar to models studied in the literature, that a linear price
sensitivity function admits such a simple optimal policy has not been
observed before, to the best of our knowledge.
%
Additionally, we show that a simple static pricing model can be
near-optimal under certain conditions. \TRtext{Specifically, we derive
universal upper bounds on the objective value attainable by any
policy, and show that static pricing can achieve a performance that is
`close' to the upper bounds in many cases.}

\subsubsection{Queueing System}

Compared to the loss system, with the queueing system, we consider
general customer and server arrival distributions, and a discrete time
setup. The objective function is to maximize the platform's profit
subject to the customer queue being stable. Towards that end, we
propose an algorithm that charges a static price, and uses a bi-modal
matching strategy with a single threshold $U$, to decide the number of
customer-server pairs that should be matched at each time slot. The
main idea of the algorithm is to choose a static price thereby making
the customer arrival rate constant over time.  The threshold $U$ is
chosen in such way that the rate of customer departures (number of
customer-server matched pairs) in each time slot is close to the
revenue-optimal departure rate. Small perturbations in the departure
rate, required to maintain the customer queue at a steady level,
result in near-optimal performance.

In particular, we show that the proposed strategy has an additive
sub-optimality gap of $O\left(\frac{\log U}{U}\right)$. Moreover, the
expected delay experienced by customers with the proposed strategy is
$O(U)$. Thus, even though we do not explicitly include the expected
delay experienced by the customers in the objective function and only
enforce the constraint that the customer queue should be stable, a
by-product of the proposed algorithm and its analysis is that the
expected delay experienced by the customers is also controlled by the
threshold~$U.$ Thus, given additional QoS constraints on the expected
delay experienced by the customers, we can choose the threshold $U$ to
achieve a tradeoff between the sub-optimality gap and the expected
delay.

\conftext{Due to space constraints, several proofs have been omitted;
  these can be found in~\cite{FullVersion}.}

\section{Loss model}
\label{sec:loss}

In this section, we consider the case where customers do not wait,
i.e., customers that arrive when there are no servers in the system
get lost/blocked.

\subsection{Model and Preliminaries}

We model the loss system in continuous time. Servers arrive to the
platform as per a Poisson process with rate $\lambda.$ Upon arrival,
servers wait in an infinite buffer queue, until they are matched to a
customer. The customer arrival rate $\mu = g(p)$ is a function of the
platform's posted price~$p,$ where $g$ is a strictly decreasing and
concave function of $p.$ Specifically, when the platform posts
price~$p,$ the time until the next customer arrival is assumed to be
exponentially distributed with mean $1/g(p).$ Thus, the function~$g$
captures the price sensitivity of the customer base.

We assume that the price $p$ is constrained to lie in
$[p_{\min},p_{\max}],$ where $0 < p_{\min} < p_{\max},$ and that the
price defaults to $p_{\max}$ when there are no servers
available. Customers that arrive when there are no servers are
blocked/lost, i.e., customers do not queue.\footnote{The case where
  customers also can get queued will be considered in
  Section~\ref{sec:queueing}.}

Under the above model, given the memorylessness in server and customer
interarrival times, the state of the system is captured by the number
of servers in the queue. Denoting the platform's price when there
are~$i$ servers in the queue by $p_i,$ the state evolves as per the
(controlled) birth-death Markov chain depicted in Figure~\ref{fig:BD}.
Here, $\mu_i := g(p_i)$ and
$\mu_i \in [g(p_{\max}),g(p_{\min})].$

\begin{figure}[h]
\centering
\begin{tikzpicture}[level/.style={sibling distance=50mm/#1}]
\node[circle,draw] (s) at (0,0) {$0$};
\node[circle,draw] (a) at (2,0) {$1$};
\node[circle,draw] (b) at (4,0) {$2$};
\node[circle,draw] (c) at (6,0) {$3$};

\draw[bend right, dashed,<-]  (s) to node [below] {$\mu_1$} (a);
\draw[bend right, dashed,<-]  (a) to node [above] {$\lambda$} (s);

\draw[bend right, dashed,<-]  (a) to node [below] {$\mu_2$} (b);
\draw[bend right, dashed,<-]  (b) to node [above] {$\lambda$} (a);

\draw[bend right, dashed,<-]  (b) to node [below] {$\mu_3$} (c);
\draw[bend right, dashed,<-]  (c) to node [above] {$\lambda$} (b);

\draw[dotted,thick]  (7,0) to node [above] {} (8,0);


\end{tikzpicture}
\caption{Birth death chain for server evolution}
\label{fig:BD}
\end{figure}
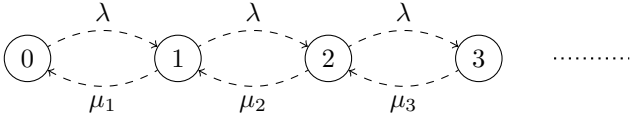

The goal of the platform is to set prices $(p_i,\ i \in \N)$ so as to
maximize
$\tilde{C} = \lambda \Exp{P} - \tilde{w} \Exp{N},$
where $P$ denotes the stationary price seen by a (matched) customer,
$N$ denotes the stationary number of servers in the system, and
$\tilde{w}$ is a positive weight. The impicit constraint here is of
course that the server queue is stable, i.e., the Markov chain
describing the temporal evolution of the number of waiting servers is
positive recurrent. The first term in~$\tilde{C}$ is the revenue rate
for the platform, the second may be interpreted as the (holding) cost
associated with servers idling. By Little's law, maximizing
$\tilde{C}$ is equivalent to maximizing
$$C = \frac{\tilde{C}}{\lambda} = \Exp{P} - \tilde{w} \Exp{T} =
\Exp{P} - w \Exp{N}$$ where $T$ denotes the stationary server sojourn
time, $w = \frac{\tilde{w}}{\lambda}.$

The optimal policy can be computed numerically using standard
machinery from the theory of Markov decision process (see, for
example, \cite{Puterman2014}). Related problem formulations are also
analysed in \cite{paschalidis2000congestion,ata2006dynamic}; these
references characterize structural properties of the optimal
policy. However, our goal here is to consider a relaxed version of the
above objective, which admits a more explicit analysis.

We now describe our relaxed objective. For a given policy
$(p_i,\ i \in \N),$ let the stationary distribution associated with
the server occupancy be denoted by $\pi = (\pi_i,\ i \in \N).$ Our
objective~$C$ may be expressed in terms of $\pi$ as follows:
\begin{align*}
  C &= \underbrace{\sum_{i=1}^{\infty} \pi_{i-1} p_i}_{\Exp{P}} - w \underbrace{\sum_{i=1}^{\infty} i \pi_i}_{\Exp{N}}.
\end{align*}
It is important to note that in the expression for $\Exp{P},$ we have
$\pi_{i-1}$ multiplying $p_i$ because the long run fraction of
(server) departures out of state~$i$ (i.e., departures that
\emph{leave behind}~$i-1$ servers in the system) equals the long run
fraction of server arrivals that \emph{see} $i-1$ servers in the
system, which, by PASTA,
equals~$\pi_{i-1}.$ Our relaxed objective is now stated as follows.
\begin{align}
  C_{\text{rel}} &=  \sum_{i=1}^{\infty} \pi_{i-1} p_{i-1} - w \sum_{i=1}^{\infty} i \pi_i = \sum_{i=0}^{\infty} \pi_{i} p_{i} - w \sum_{i=1}^{\infty} i \pi_i
  \label{eq:relaxed_obj}
\end{align}
Note that under this relaxation, the term
$\sum_{i=1}^{\infty} \pi_{i-1} p_i$ in the objective is replaced by
the more tractable term $\sum_{i=1}^{\infty} \pi_{i-1} p_{i-1}.$ The
motivation for this relaxation is of course to align the distribution
used to average price with the stationary distribution.\footnote{This
  `mis-alignment' arises in~$C$ because we control the \emph{left}
  transition rates in this model. In an alternative model wherein the
  control is on the \emph{rightward} transition rates (this would
  arise if customers were to queue, and servers do not; think of an
  airport taxi lot), the `mis-alignment' would not occur, and we would
  not need to relax the objective in this manner.} Clearly,
$C_{\text{rel}}$ and $C$ would be close when the price varies slowly
with state. Specifically, for static pricing policies, the two
objectives are identical. However, even for the bang-bang type
policies we consider later, we find that two objectives are aligned,
as we demonstrate as part of our numerical experiments in
Section~\ref{sec:loss_numerics}.

Most of our results will be derived for the case where the price
sensitivity function~$g$ is linear:
\begin{assumption}
  \label{ass:linear}
  $g(p)$ depends linearly on~$p,$ i.e., $g(p) = \beta - \alpha p,$
  where $\beta,\alpha > 0.$ Moreover,
  $\mu_{\max}:= \beta - \alpha p_{\min}> \lambda,$ and
  $\mu_{\min} := \beta - \alpha p_{\max} > 0.$
\end{assumption}
We also consider a more general class of concave price sensitivity
functions, and prove bounds on the sub-optimality of simple static
pricing policies for this class.
\begin{assumption}
  \label{ass:concave}
  $g(p) = (\beta - \alpha p)^{\theta},$ where $\theta \in (0,1],$ and
  $\beta,\alpha > 0.$ Moreover,
  $\mu_{\max}:= (\beta - \alpha p_{\min})^{\theta}> \lambda,$ and
  $\mu_{\min} := (\beta - \alpha p_{\max})^{\theta} > 0.$
\end{assumption}
Under both assumptions, note that $p_{\max} < \frac{\beta}{\alpha}.$
Thus, $\frac{\beta}{\alpha}$ is a trivial upper bound on the objective
value achievable. More refined bounds will be derived in
Section~\ref{sec:static}.

\hide{
The remainder of this section is organised as follows. We first show
that the optimal policy under linear price sensitivity and our relaxed
objective is of bang-bang type. This is established as a consequence
of a bang-bang lemma for Markov chains, which is discussed next. We
then consider static pricing, and show that under certain conditions,
simple static pricing has a bounded sub-optimality. Finally, we
conclude with some numerical results for the loss model.
}

\subsection{Linear price sensitivity \& relaxed objective: Optimal policy}
\label{sec:bang_bang_optimal}

Throughout this section, we make Assumption~\ref{ass:linear}, and
consider the relaxed objective~$C_{\text{rel}}$ (see
\eqref{eq:relaxed_obj}). Recall that the number of servers evolves as
per the birth death Markov chain shown in Figure~\ref{fig:BD}. Here,
$p_i,$ or equivalently $\mu_i = g(p_i),$ can be interpreted as the
`action' taken in state~$i.$ Recall also that $p_0 = p_{\max};$ let
$\mu_0 := g(p_0) = \mu_{\min}.$

We begin by rewriting $C_{\text{rel}}$ as follows.
\begin{lemma}
  \label{lemma:relaxed_obj}
  Under Assumption~\ref{ass:linear}, 
  $C_{\text{rel}} = \frac{1}{\alpha} \left(\beta - \pi_0 \mu_0 -\lambda \right)
  - w \sum_{i=1}^{\infty} i \pi_i.$
\end{lemma}

Lemma~\ref{lemma:relaxed_obj} states that under linear price
sensivity, the first term in the relaxed objective decreases linearly
in $\pi_0.$ In other words, maximizing the first term boils down to
minimizing $\pi_0.$ We now exploit this property to show that the
optimal policy is of bang-bang type.

\begin{theorem}
  \label{thm:bangbangpricing}
  Under Assumption~\ref{ass:linear} (wherein the price sensitivity
  function~$g$ is linear), there exists a policy that optimizes
  $C_{\text{rel}}$ of the form
  $$p_i = \left\{
  \begin{array}{cl}
    p_{\max}& \text{ for } i < \ell^*\\
    p^* & \text{ for } i = \ell^*\\
    p_{\min} & \text{ for } i > \ell^*
  \end{array}\right.,$$
where $\ell^* \in \N \cup \{\infty\},$ $p^* \in [p_{\min},p_{\max}].$
\end{theorem}
The bang-bang policy stated in Theorem~\ref{thm:bangbangpricing} has
the following interpretation: When there are fewer than $\ell^*$
servers in the system, the platform sets the maximum price, thereby
limiting customer arrivals as far as possible so as to provide the
maximum upward drift to the server queue.  On the other hand, when the
number of servers exceeds $\ell^*,$ the platform sets the minimum
price, thereby maximizing the rate of customer arrivals and providing
the maximum downward drift to the server queue. Finally, when there
are exactly $\ell^*$ servers, the platform sets an intermediate
price~$p^*.$ Thus, the platform in effect seeks to maintain the number
of servers at around~$\ell^*$ to the extent that its control over
customer arrivals permits. As will be apparent from the proof of
Theorem~\ref{thm:bangbangpricing}, the aforementioned policy
effectively minimizes the probability that the server queue becomes
empty, subject to an upper bound on the expected stationary number of
servers in the system.

\begin{remark}
  It is important to note that the bang-bang structure of the optimal
  policy as stated in Theorem~\ref{thm:bangbangpricing} is
  significantly stronger than the monotonicity properties that are
  typically established in structured MDPs (see, for
  example,~\cite{serfozo1976}). Moreover, note that
  Theorem~\ref{thm:bangbangpricing} is proved using the representation
  of the relaxed objective in Lemma~\ref{lemma:relaxed_obj}, which in
  turn relies heavily on the linearily of the price sensitivity
  function~$g.$
\end{remark}

\begin{remark}
  While the optimal policy as stated in
  Theorem~\ref{thm:bangbangpricing} is parameterized by
  $(\ell^*,p^*),$ it can also be parameterized via a single
  parameter~$x^* \in \mathbb{R}_+.$ The two parameterizations are
  related as follows: $\ell^* = \ceil{x^*},$ and $p^* = p_{\max} -
  (\ceil{x^*} - x^*)(p_{\max}-p_{\min}).$ This one-dimensional
  parameterization simplifies the task of computing the optimal policy
  as a function of the system parameters, and is amenable to efficient
  online learning via standard stochastic approximation techniques
  (see~\cite{Roy2019}).
\end{remark}

\begin{proof}[Proof of Theorem~\ref{thm:bangbangpricing}]
  It suffices to show that an optimal policy of the specified form
  exists for any \emph{constrained} optimization of the following
  form, parameterized by $\C > 0.$
  \begin{align*}
    \max\ \frac{1}{\alpha} \left(\beta - \pi_0 \mu_0 -\lambda \right) \qquad 
    s.t.\ \sum_{i=1}^{\infty} i \pi_i \leq \C
  \end{align*}
  Equivalently, it suffices to show that an optimal policy of the
  specified form exists for any \emph{constrained} optimization of the
  following form:
  \begin{equation}
    \min\ \pi_0  \qquad
    s.t.\ \sum_{i=1}^{\infty} i \pi_i \leq \C
          \label{opt:Markov}
  \end{equation} 
  That this problem, which is well motivated in its own right, has an
  optimal solution of the bang-bang form specified, is proved as
  Lemma~\ref{lemma:bangbang}.
\end{proof}

The optimization~\eqref{opt:Markov} is a natural and well-motivated
problem in the context of controlled birth death Markov chains
over~$\Z_+$, where the goal is to minimize the stationary probability
of state~0, subject to an upper bound on the expectation of the steady
state distribution. Since this problem is interesting in its own
right, we study it (independently of its application to our dynamic
pricing model) in the following section.

\subsection{A bang-bang lemma for controlled Markov chains}

Consider a birth death Markov chain over state space $\Z_+.$ The
transition rate from state $i$ to $i+1$ is denoted by $\lambda_i$ for
$i \geq 0$ and the transition rate from state $i$ to $i-1$ is denoted
by $\mu_i$ for $i \geq 1.$ For $i \geq 1$
$\rho_i := \frac{\lambda_{i-1}}{\mu_i}.$

Assuming the chain is positive recurrent, the stationary distribution
is given by $\pi_i = \frac{h_i}{\sum_{j=0}^{\infty} h_j},$ where
$h_0 = 1,$ $h_i = \prod_{j=1}^i \rho_j$ for $i \geq 1.$

With $\rho_i$ constrained to lie in $[\underline{\rho},\bar{\rho}],$
where $0 < \underline{\rho} < \bar{\rho},$ our goal is to minimize
$\pi_0$ subject to a moment condition on the stationary
distribution. Formally, this is posed as:
\begin{equation}
  \begin{array}{rl}
    \min.& \pi_0 \\
    s.t.& \sum_{i=0}^{\infty} i \pi_i \leq C \\
        & \rho_i \in [\underline{\rho},\bar{\rho}] \quad (i \geq 1)
  \end{array}
  \label{eq:bang_bang_obj}
\end{equation}
An implicit constraint here is of course that the chain is positive
recurrent. Clearly, $\underline{\rho} < 1$ is necessary and
sufficient for the feasibility of positive recurrence of the
chain. Note also that if $\underline{\rho} < 1$ and
$\frac{\underline{\rho}}{1 - \underline{\rho}} > C,$ then the
optimization~\eqref{eq:bang_bang_obj} is infeasible (since the
expected steady state value of the chain is easily seen to be at least
$\frac{\underline{\rho}}{1 - \underline{\rho}}$). Thus, the
optimization~\eqref{eq:bang_bang_obj} is well posed if
$\underline{\rho} < 1$ and
$\frac{\underline{\rho}}{1 - \underline{\rho}} \leq C.$ Under these
conditions, the following lemma shows that the optimal control is of
bang-bang type.
\begin{lemma}
  \label{lemma:bangbang}
  Assuming $\underline{\rho} < 1$ and
  $\frac{\underline{\rho}}{1 - \underline{\rho}} \leq C,$ the optimal
  solution of \eqref{eq:bang_bang_obj} is of the form
  $$\rho_i = \left\{
  \begin{array}{cl}
    \bar\rho & \text{ for } i < \ell^*\\
    r^* & \text{ for } i = \ell^*\\
    \underline{\rho} & \text{ for } i > \ell^*
  \end{array}\right.,$$
where $\ell^* \in \N \cup \{\infty\},$ $r^* \in [\underline{\rho},\bar{\rho}].$
\end{lemma}

This lemma is proved by showing that any feasible policy that does not
have the above bang bang structure can be improved upon via a
perturbation towards this structure (details can be found
in~\cite{FullVersion}). The proof arguments are similar to those
in~\cite{KiranKavithaNair}, where a bang-bang style policy is shown to
be optimal in a different context.

\hide{After deriving structural results on the optimal policy for the
  customer arrival rate that is a linear function of the price in this
  section, next, we consider simple pricing strategies, and bound
  their competitive ratios. Recall that competitive ratio is defined
  as the ratio of the profit ($C_{\text{rel}}$) of a strategy and the
  profit of the optimal strategy, and a strategy is defined to be
  $c$-competitive if the ratio of the profit of the strategy and that
  of the optimal strategy is at least $1/c$.}

\subsection{Single (static) price policies}
\label{sec:static}

Next, we consider the simplest pricing policy: a constant,
state-independent price, under the relaxed
objective~$C_{\text{rel}}.$\footnote{Note that under static pricing,
  $C$ and $C_{\text{rel}}$ are equal. However, we since one of our
  universal upper bounds (Lemma~\ref{lemma:ub1}) is only proved for
  the relaxed objective, we persist with the use of $C_{\text{rel}}$
  throughout this section.}  While static pricing might seem naive, we
show that under certain conditions, if the static price is chosen
carefully, the suboptimality relative to the (unknown) optimal policy
can be bounded. This is somewhat analogous to what happens with the
classical server speed scaling problem, where it is known that a
suitably chosen static speed choice is constant competitive under a
stochastic workload model \cite{Wierman2012,Vaze2020}.

We begin by deriving some universal upper bounds on the objective
value under \emph{any} policy (not necessarily static).

\subsubsection{Universal upper bounds}
\label{sec:universal_bounds_loss}

\begin{lemma}
  \label{lemma:ub1}
  Under Assumption~\ref{ass:concave}, for any policy, we have
  $C_{\text{rel}} \leq g^{-1}(\lambda).$
\end{lemma}

\begin{lemma}[Light traffic bound]
  \label{lemma:ub2}
  Under Assumption~\ref{ass:concave}, for any policy,
  $C_{\text{rel}} \leq \max_{p \in [p_{\min},p_{\max}]} \left(p - \frac{\tilde{w}}{g(p)}\right).$
\end{lemma}

Next, we apply these upper bounds to study the competitiveness of
static pricing for the case of linear price sensitivity.

\subsubsection{Competitiveness under linear model}

We begin by specializing the above universal bounds to
Assumption~\ref{ass:linear}, i.e., $g(p) = \beta - \alpha p,$ where
$\beta,\alpha > 0.$
\TRtext{An application of Lemma~\ref{lemma:ub1} for this case yields:
  $C_{\text{rel}} \leq \frac{\beta - \lambda}{\alpha}.$ Next,
  application of Lemma~\ref{lemma:ub2} yields: $C_{\text{rel}} \leq
  \max_{p \in \mathbb{R}} \left(p - \frac{\tilde{w}}{\beta - \alpha
    p}\right) = \frac{\beta - 2\sqrt{\tilde{w}\alpha}}{\alpha}.$}
Combining the two bounds above together, under
Assumption~\ref{ass:linear}, for any policy,
\begin{equation}
  \label{eq:upper_bound}
  C_{\text{rel}} \leq \frac{\beta}{\alpha} - \max\left(\frac{\lambda}{\alpha},\frac{2\sqrt{\tilde{w}\alpha}}{\alpha}\right).
\end{equation}
Clearly, the condition $\beta > 2\sqrt{\tilde{w}\alpha}$ is a necessary
condition for positive objective value under any policy. Now the
optimal static price is given by
$p^* = \argmax_{p \in [p_{\min}, p_{\max}]} \left(p - \frac{\tilde{w}}{\beta - \alpha p - \lambda}\right).$
If
$\mu_{\max} = \beta - \alpha p_{\min} \geq \lambda + \sqrt{\alpha
  \tilde{w}},$ then
$p^* = \frac{\beta - \sqrt{\alpha \tilde{w}} - \lambda}{\alpha}.$ The
static policy that always chooses the price $p^*$ would then have
payoff
\begin{equation}
  \label{eq:opt_static_payoff}
  C_{\text{rel}}(p^*) = \frac{\beta - 2\sqrt{\alpha \tilde{w}} - \lambda}{\alpha}.
\end{equation}

Comparing~\eqref{eq:upper_bound} and \eqref{eq:opt_static_payoff}, it
follows that so long as
$\mu_{\max} \geq \lambda + \sqrt{\alpha \tilde{w}},$ i.e., it is
feasible to maintain a $\sqrt{\alpha \tilde{w}}$ slack in the customer
arrival rate relative to the server arrival rate, the reduction in
payoff from $\frac{\beta}{\alpha}$ under the optimal static pricing
policy is at most twice that under any policy. Reframing the objective
as a minimization of the payoff reduction from its (unattainable)
upper bound $\frac{\beta}{\alpha},$ this implies a competitive ratio
of at most two.

\subsubsection{Competitiveness under non-linear model}

We now consider the more general concave price sensitivity model
specified by Assumption~\ref{ass:concave}.
Specializing our two upper bounds to this particular choice of
$g(\cdot),$ we get that under any policy,
\begin{equation}
  \label{eq:upper_bound_concave}
  C_{\text{rel}} \leq \frac{\beta}{\alpha} -
  \max\left(\frac{\lambda^{1/\theta}}{\alpha}, \frac{B}{\alpha}
  \right),
\end{equation}
where $B := (\tilde{w}\alpha\theta)^{1/(\theta+1)} (1 + 1/\theta).$

Unlike in the linear case however, this non-linear model for
$g(\cdot)$ does not admit a closed form characterization of the
optimal static price. We thus consider two separate cases based on
which term contributes to the max in
\eqref{eq:upper_bound_concave}. For each of these cases, a different
reasonable choice of static price is considered.

\noindent {\bf Case 1 (heavy traffic):}
$\lambda^{1/\theta} \geq B$

Consider the static policy that sets the price $p$ such that
$\mu = (\beta - \alpha p)^{\theta} = \gamma \lambda,$ where
$\gamma > 1.$ Clearly, such a $\gamma$ exists, given that
$\mu_{\max} > \lambda.$ The payoff under this policy is
\begin{align*}
  C_{\text{rel}} &= \frac{\beta}{\alpha} - \frac{(\gamma \lambda)^{1/\theta}}{\alpha} - \frac{w}{(\gamma-1) \lambda}.
\end{align*}
For large $\lambda,$ the last term above would be negligible, meaning
the payoff reduction from $\frac{\beta}{\alpha}$ would be
(approximately) at most a factor of $\gamma^{1/\theta}$ of that under
any policy.

\noindent {\bf Case 2 (light traffic):}
$\lambda^{1/\theta} < B$

In this case, consider the static policy that sets the price such that
$\mu = (\beta - \alpha p)^{\theta} = \gamma B^{\theta},$ where
$\gamma > 1.$ Such a choice is of course feasible only when
$\mu_{\max} > B^{\theta}.$ If so, the cost under this policy is
bounded as:
\begin{align*}
  C_{\text{rel}} &\geq \frac{\beta}{\alpha} - \frac{\gamma^{1/\theta} B}{\alpha} - \frac{w}{(\gamma-1) B^{\theta}}.
\end{align*}

For large $B,$ the last term above would be negligible, meaning the
payoff reduction from $\frac{\beta}{\alpha}$ would be (approximately)
at most a factor of $\gamma^{1/\theta}$ of that under any policy, as
before.

\subsection{Numerical experiments}
\label{sec:loss_numerics}

Since we have used the relaxed objective~$C_{\text{rel}}$ in the
preceding sections, we now present some numerical results illustrating
the connection between $C_{\text{rel}}$ and $C.$ Specifically, we
compare both objectives over the (one-dimensional) space of bang-bang
policies of the kind we proved as optimal for $C_{\text{rel}}$ for
linear price sensitivity.

In Figure~\ref{fig:Obj_comparison}, we plot $C_{\text{rel}}$ and $C$
as a function of a single parameter~$x$ that specifies the bang bang
policy, as described in Section~\ref{sec:bang_bang_optimal}, for
different values of~$w.$ We note that increasing the weight~$w$ on the
holding cost decreases the objective values, and decreases the optimal
choice of $x,$ as expected. However, what is interesting to note is
that the optimal choice of $x$ under the relaxed objective matches
almost perfectly with the optimal choice under the original
objective. This suggests that the optimal bang-bang policy under the
relaxed objective is also a near optimal choice (within the class of
bang-bang policies) for our original objective.

\begin{figure} [t]
  \centering
  \subfigure [$w = 0.05$]
  {\includegraphics[width=0.4\columnwidth]{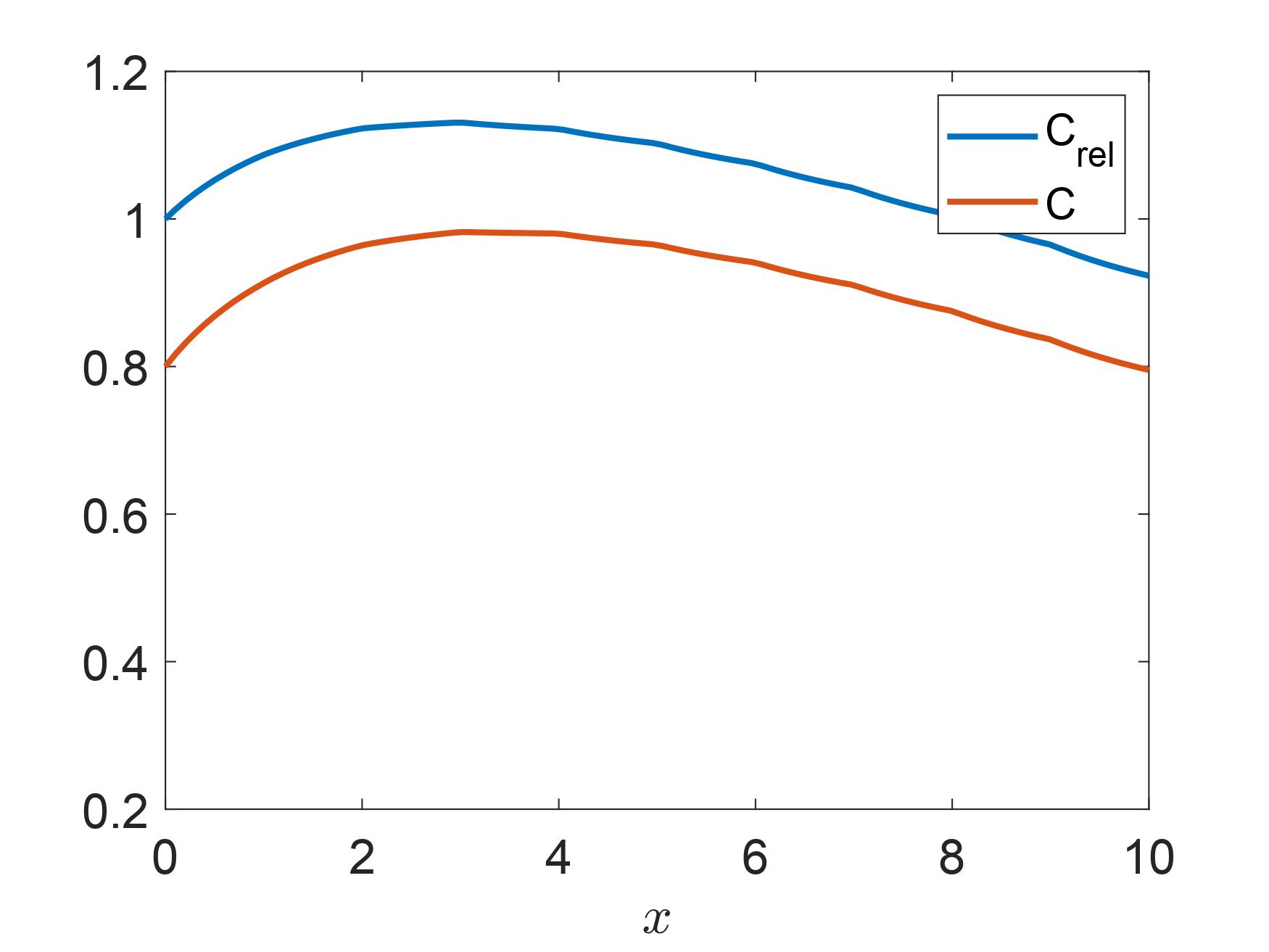}
  }
    \hspace{2 mm}
  \subfigure [$w = 0.1$]
  {\includegraphics[width=0.4\columnwidth]{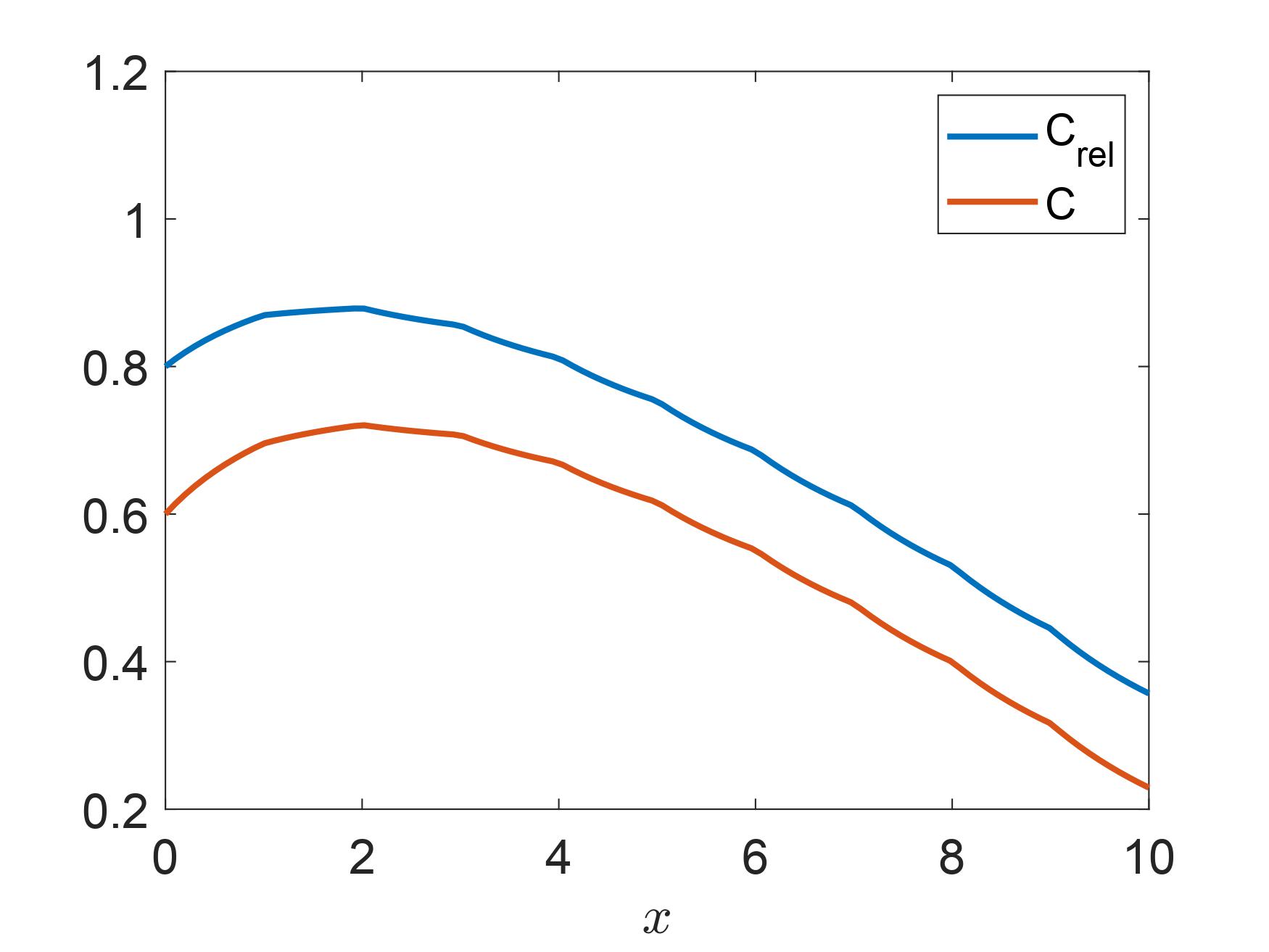}
  }
  \caption{Comparison of the original objective~$C$ and the relaxed
    objective~$C_{\text{rel}}$ for bang-bang policies. The policy parameters
    are related to the $x$-axis label as
    follows: $\ell = \ceil{x},$ $p = p_{\max} - (\ceil{x}-x)(p_{\max}-p_{\min}).$
    Chosen system parameters are as follows:
    $\lambda = 2,$ $p_{\min} = 1,$ $p_{\max} = 2,$ $\beta = 3.5,$ $\alpha = 1.$ 
    \label{fig:Obj_comparison}}
\end{figure}

\section{Queueing Model}
\label{sec:queueing}

Consider a discrete time system,  where the arrival rate for servers is insensitive to price and is equal to $\lambda(t) = \lambda$ for each time $t$\footnote{With abuse of notation, we are calling $\lambda$ as the arrival rate with a discrete time system similar to the continuous time loss model of Section \ref{sec:loss}.}. The customers respond to price set by the platform, and the arrival rate of customers in any time slot $t$ is $\mu(p(t))$ if the price chosen by the platform is $p(t)$ at time $t$. 
We assume the natural model, where $\mu(p)$ is a non-increasing continuous function of $p$. In addition, we assume that  $\mu(p)$ is a concave function, and  $p \mu(p)$ is a unimodal function of $p$, following prior work \cite{sood2018pricing}.
Recall that for notational convenience in Section \ref{sec:loss}, we assumed that the customer arrival rate is $\mu=g(p)$, while in this section, we use simply $\mu(p)$.

Under this model, let $A(t)$ and $B(t)$ be the number of servers and
customers that arrive in time slot $t$, following i.i.d. processes,
with rate $\lambda$ and $\mu(p(t))$, respectively. In particular,
$\lambda = \bbE\{A(t)\}$ and $\mu(p(t)) = \bbE\{B(t)\}$.  Let $M(t)$
be the set of customers that depart (number of customer-server pairs
that are matched) at time $t$. Since there are two queues
corresponding to servers and customers, making both queues stable
(that allows the existence of steady state distribution)
simultaneously is not possible without some exogenous constraint
(unless we have a separate price lever for each queue, as is done
in~\cite{mahavir2020dynamic}). So to keep the server queue stable, we
consider an upper limit ${\bar S}$ on the number of outstanding
servers, while ensuring the stability of customer queue will be part
of the considered problem. Thus, a server arriving when the server
queue size is ${\bar S}$ is not admitted.

 The number of servers $S(t)$, and
customers $C(t)$, in the system at time $t$, evolve as follows.
\begin{align}\label{eq:qupdate1}
  S(t+1) & = \min\{(S(t) + A(t) - |M(t)|)^+, {\bar S}\}, \\\label{eq:qupdate2}
  C(t+1) &  = (C(t) + B(t) - |M(t)|)^+,
\end{align}
where $|M(t)|$ is the number of customer-server pairs that are matched
by the platform in time slot $t$, and $(x)^+ = \max\{0,x\}$. 
The evolution of $S(t)$ and $C(t)$ is hence coupled via $M(t)$.  

\begin{remark}Compared to Section \ref{sec:loss}, where we considered a loss model, in this section, customers wait in the queue, and depart only when they are matched to any server. 
Moreover, we are considering a more general system than Section
\ref{sec:loss}, that evolves in discrete time, and the customer and
server arrival processes are not restricted to follow a Poisson
distribution.
\end{remark}

Customer $i$ arriving at time $t$ sees or commits to price $p_{i} = p(t)$, and let the platform make profit of $p_{i}$ when customer $i$ is matched (departs) to some server at time $t'\ge t$. 
Thus, the profit made by the platform at time $t$ is determined by the set of $M(t)$ customers that depart at time $t$, and the price they saw when they arrived $p_i, i \in M(t)$.
Thus, given the price $p(t)$ chosen by the platform, its profit  (that is a
function of price $p(t)$) is given by
\begin{equation}\label{eq:profit}
\bV = \lim_{T\rightarrow \infty}\frac{1}{T}\sum_{t=1}^TV\left(\sum_{i\in M(t)} p_i\right), \end{equation}
where $V$ is a non-decreasing concave function. 

Note that in defining the platform's profit \eqref{eq:profit}, we have not directly accounted for the payment to the servers, however, that is implicitly captured by assuming that the platform keeps a constant factor of the profit to itself and distributes the rest uniformly across all servers. Thus, maximizing profit, is equivalant to maximizing the payment to the servers. Servers not being incentivised per-customer matching is well justified following \cite{banerjee2015pricing}, which shows that servers payoffs are better modelled as long-term rewards.

The optimization problem that we consider is as follows. \begin{equation}\label{eq:profit1}
\max_{p(t), M(t)}\quad \bV \quad \text{s.t.} \ \underbrace{\lim_{T\rightarrow \infty}\frac{1}{T}\sum_{t=1}^T\bbE\{C(t)\} < \infty}_{\text{stability of 
customer queue}}.
\end{equation}
The stability condition in \eqref{eq:profit1} takes care of the fact
that the delay seen by arriving customers is bounded.

\begin{remark}
  An alternate formulation to \eqref{eq:profit1} is to maximize $\bV$
  subject to a constraint on the expected delay seen by the customers,
  or maximize a linear combination of $\bV$ and expected delay seen by
  the customers. Both these alternatives are, however, more
  challenging to solve in the setting considered, where we are not
  considering any scaling limit regime,
  unlike~\cite{banerjee2015pricing} and similar papers. In what will
  follow, the algorithm we propose to solve \eqref{eq:profit1}, will
  have a parameter that will tradeoff the sub-optimality gap (Theorem
  \ref{thm:lbq}) and the expected delay (Lemma \ref{lem:delay}) seen
  by the customers. Thus, given a constraint on the expected delay
  seen by the customers, we can tune the parameter and bound the
  sub-optimality gap.

\end{remark}


Let the  {\bf optimal price} be $p^\star(t)$ and {\bf optimal matching decision} be $M^\star(t)$ to maximize \eqref{eq:profit1} under the stability constraint, and let the optimal profit be $\bV_\opt$.
Next, we upper bound  $\bV_\opt$. Towards that end, we define a {\bf critical} quantity $p^\star$, as follows.
\begin{equation}\label{eq:pstar}
p^\star =\arg \max_{p, \lambda \ge \mu(p)} p\cdot \mu(p).
\end{equation}

Note that since $\mu(p)$ is assumed to be a non-increasing continuous concave function of $p$, it follows that  $p\mu(p)$ is a concave function. Using this fact, we get the following upper bound on $\bV_\opt$.

%
\begin{lemma} \label{lem:ub}
$ \bV_\opt  \le V(p^\star \mu(p^\star)).$
\end{lemma}
%




Lemma \ref{lem:ub} essentially says that the largest profit is
possible if the price is set as constant $p^\star$, and the number of
matched customer-server pairs $|M(t)|$ is a constant equal to the
expected customer arrival rate $\mu(p^\star)$ at price $p^\star$. The
proof follows from a straightforward application of Jensen's
inequality, given that $V$ is a concave function.
 
 Next, we propose an algorithm for setting the price $p(t)$, and choosing the number of matched customers $|M(t)|$  in each time slot $t$, and lower bound its profit.
 
{\bf Algorithm} $\cA$: Following the definitions of S(t) \eqref{eq:qupdate1} and C(t) \eqref{eq:qupdate2}, let $N(t) = \min\{S(t), C(t)\}$. Solving \eqref{eq:pstar}, either we have $\lambda > \mu(p^\star)$ or $\lambda = \mu(p^\star)$.

 If $\lambda > \mu(p^\star)$, the algorithm chooses constant price
 $p(t) = p^\star$, $\forall \ t$, defined in \eqref{eq:pstar}, and the
 number of customer-server pairs {\bf matched} in slot $t$ using FIFO
 schedule are
\begin{equation}\label{eq:alg}
|M_\cA(t)| = \begin{cases} 0  & \text{if} \ N(t) <  \mu(p^\star) - \delta, \\
\mu(p^\star) - \delta & \text{if} \ \mu(p^\star) - \delta \le N(t) \le U/2,\\
 \mu(p^\star) + \delta &\text{if} \ N(t) > U/2,
\end{cases}
\end{equation} 
where $U < {\bar S}$ is some threshold, and $\delta >0$ that will be chosen later.  Threshold $U$ will control both the profit made by the platform as well as the expected waiting time of any customer.

If $\lambda = \mu(p^\star)$, then $p(t) = p^\star + \epsilon$ (to ensure that the customer arrival rate is lower than the server arrival rate), while choice of $|M_\cA(t)|$ remains unchanged as in \eqref{eq:alg}. 

For the rest of this section, we consider the case when
$\lambda > \mu(p^\star)$. All results will go through even when
$\lambda = \mu(p^\star)$ with an additional $O(\epsilon)$
penalty. 

\begin{remark}Note that we are not enforcing the integrality constraint on $|M_\cA(t)|$, similar to prior works \cite{lin2012dynamic, berg2019hesrpt} on dynamic decision problems with server-customer queues, where the consideration is that the number of servers/customers is large enough at an aggregate scale, and $|M_\cA(t)|$ can be thought of as the fraction of customers served. 
\end{remark}

With algorithm $\cA$, for both cases,  $\lambda > \mu(p^\star)$ or 
$\lambda = \mu(p^\star)$, the arrival rate of servers is more than the
arrival rate of customers with the
algorithm, ensuring stability of the customer queue, satisfying the constraint in \eqref{eq:profit1}. 
Thus, we only need to derive a lower bound on the profit of algorithm $\cA$, for which we need the following definition.
Let $\sigma^2_C = \lim_{T\rightarrow \infty} \frac{1}{T} \text{var}\left(\sum_{t=1}^T B(t)\right)$ be the variance of process $B(t)$, number of arrivals of customers at time $t$.
\begin{remark}
With the algorithm $\cA$ that charges price $p(t) = p^\star$, the process $B(t)$ (the number of customer arrivals with mean $\mu(p^\star)$) is an i.i.d. process. Therefore, $\sigma^2_C$ is well-defined.
\end{remark}

The main result of this section is the following.

\begin{theorem}\label{thm:lbq}
Choosing $\delta = \alpha\left( \frac{\log U}{U}\right),$ where $\alpha = \beta \sigma^2_C$ and $\beta\ge 2$, for algorithm $\cA$, 
  $$\bV_\cA\ge \bV_\opt - {\cal O}\left(\left(\frac{\log U}{U}\right)^2\right).$$
\end{theorem}
Thus, algorithm $\cA$ is near-optimal and the sub-optimality gap is governed by the choice of threshold $U$.

\emph{Proof Sketch:} Let $N$ be distributed as as per the steady state
distribution of $N(t)$ with algorithm $\cA$.  By definition, algorithm
$\cA$ can achieve profit close to $V(p^\star \mu(p^\star))$ (where
$\bV_\opt \le V(p^\star \mu(p^\star))$) as long as $N > \mu(p^\star) -
\delta$.  So the main result is to show that for $\cA$,
$P_{\text{outage}} = \bbP(N < \mu(p^\star) - \delta ) = {\cal
  O}\left(U^{-\beta}\right)$ for $\delta = \alpha\left( \frac{\log
  U}{U}\right),$ where $\alpha = \beta \sigma^2_C$ and $\beta\ge 2$,
which we prove in Lemma \ref{lem:hittingprob}. \qed

\begin{figure}[!h]
    \centering
    \includegraphics[scale=0.35]{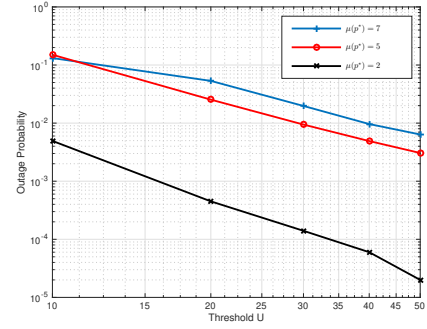}
    \caption{ Outage probability with Poisson driver and customer arrivals with different values of $\mu(p^\star)$ and $\Delta=.1$, respectively, as a function of $U$.}
    \label{fig:outage}
\end{figure}

\begin{figure}[!h]
    \centering
    \includegraphics[scale=0.35]{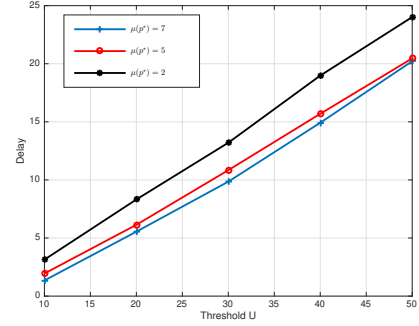}
    \caption{Expected delay experienced by a customer with Poisson driver and customer arrivals with different value of $\mu(p^\star)$ and $\Delta=.1$, respectively, as a function of $U$.}    \label{fig:delay}
\end{figure}

\begin{lemma}\label{lem:hittingprob}
  For algorithm $\cA$ \eqref{eq:alg} with price $p(t) = p^\star, \ \forall \ t$, $$P_{\text{outage}}=\bbP(N <  \mu(p^\star) - \delta ) = {\cal O}\left(U^{-\beta}\right),$$ by choosing $\delta = \alpha\left( \frac{\log U}{U}\right),$ where $\alpha = \beta \sigma^2_C$ and $\beta\ge 2$. 
\end{lemma}
Similar to the Lemma \ref{lem:hittingprob}, we can get an upper bound on the $P(N \ge x)$ for $x\ge U$ as follows. 
\begin{lemma}\label{lem:tail}
  For algorithm $\cA$ \eqref{eq:alg} with price $p(t) = p^\star, \ \forall \ t$, $$\bbP(N \ge x ) = {\cal O}\left(U^{-\beta}\right)$$ by choosing $\delta = \alpha\left( \frac{\log U}{U}\right)$, where $\alpha = \beta \sigma^2_C$ and $\beta\ge 2$, and $x>U$. 
\end{lemma} 
Proof of Lemma \ref{lem:tail} is identical to that of Lemma \ref{lem:hittingprob} and is omitted. An important consequence of Lemma \ref{lem:tail} is a bound on the expected customer queue length, which using Little's law gives a bound on the average delay seen by the customers.
\begin{corollary} $\bbE\{N\} = {\cal O}(U)$ by choosing $\beta \ge 3$.
\end{corollary}
Lemma~\ref{lem:tail} helps us in bounding the expected waiting time $\bbE\{W\}$ seen by any customer using Little's Law. Recall that with algorithm $\cA$, the customer arrival rate is constant $\mu(p^\star)$, and the system is stable. Thus, the expected customer departure rate is also $\mu(p^\star)$. Thus, using Little's Law, \begin{equation}\label{eq:delay}
\bbE\{W\} =  \bbE\{N\}/\mu(p^\star) = {\cal O}\left(\frac{U}{\mu(p^\star)}\right).
\end{equation}
Thus, we have proved the following lemma.
\begin{lemma}\label{lem:delay}
  With algorithm $\cA$, the expected waiting time $\bbE\{W\}$ seen by any customer is 
  $\bbE\{W\} ={\cal O}\left(\frac{U}{\mu(p^\star)}\right).$
\end{lemma}

{\it Discussion:} The upshot of Theorem \ref{thm:lbq} is that near optimal profit can be obtained with static pricing, and advantage of dynamic pricing is arbitrarily small. Similar conclusions have been derived in earlier papers e.g., \cite{banerjee2015pricing}, with an objective function that is a sum of the platform profit and the expected delay seen by the servers, however, these results were shown in limiting regimes, such as scaling the number of customers and servers, and considering a fluid limit. Moreover, Lemma \ref{lem:delay} shows that the static pricing policy that achieves near optimal profit has bounded expected delay for customers, as a function of the parameter $U$ that also controls the sub-optimality gap. 
Thus, threshold level $U$ provides a tradeoff between the sub-optimality gap from Theorem \ref{thm:lbq} and the expected delay seen by a customer \eqref{eq:delay}, and can be appropriately chosen given QoS requirements.

In Figs. \ref{fig:outage} and \ref{fig:delay}, we illustrate the tradeoff between the sub-optimality gap and the expected customer delay as a function of $U$,
when both the customer and the server arrival distributions are Poisson. 
Recall that the sub-optimality gap is essentially controlled by the outage probability. Thus, we plot the outage probability and expected customer queue length (that controls the expected delay), as a function of $U$ for different values of $\mu(p^\star)$ and $\Delta = \lambda - \mu(p^\star)=0.1$. 
This way we  avoid making specific choices of $V$ and $\mu$, while still capturing the quantities of interest. We observe that the simulated performance remains unchanged for higher values of $\Delta$, since the minimum of the two rates $\lambda$, $\mu(p^\star)$, controls the performance. Hence to avoid cluttered plots, we only illustrate the $\Delta=0.1$ case. 
For simulations, we use $\beta = 2$, $\sigma_C^2 = 2$ that completely defines algorithm $\cA$'s choice of $\delta$.
As promised by theory, the outage probability falls off as power-law with respect to $U$, while the expected delay is close to $U/2$.

\section{Conclusions}
In this paper, we have considered simple pricing strategies for two
sided queue matching problems, that have typically been either
analysed in large system limits, or for which optimal but non-closed
form policies are known. We considered both the loss and the queueing
systems, which are relevant for practical applications. For the loss
system, we showed an important structural result that the optimal
policy is of the bang bang type when the customer arrival rate is a
linear function of the price.
For the queueing model, we propose a simple static pricing strategy,
and a bi-modal matching decision that is shown to be near-optimal,
together with a bound on the expected delay seen by the customers.


\bibliographystyle{IEEEtran}
\bibliography{Refs}

\onecolumn
\section{Appendices}

\subsection{Proof of Lemma~\ref{lemma:relaxed_obj}}
\begin{proof}
  Under Assumption~\ref{ass:linear}, we rewrite the first term
  of~$C_{\text{rel}}$ as follows.
  \begin{align*}
    \sum_{i=0}^{\infty} \pi_{i} p_{i} &= \frac{1}{\alpha} \sum_{i=0}^{\infty} \pi_i (\beta - \mu_i) \\
                                      &= \frac{\beta - \pi_0 \mu_0 - \sum_{i=1}^{\infty} \pi_i \mu_i }{\alpha} \\
                                      &\stackrel{(a)}{=} \frac{\beta - \pi_0 \mu_0 - \sum_{i=1}^{\infty} \pi_{i-1} \lambda }{\alpha} \\
                                      &= \frac{\beta - \pi_0 \mu_0 - \lambda}{\alpha}.    
  \end{align*}
  Step$(a)$ follows from the reversibility of the birth death Markov
  chain, which gives $\pi_{i-1} \lambda = \pi_i \mu_i$ for $i \geq 1.$
\end{proof}

\subsection{Proof of Lemma~\ref{lemma:bangbang}}
This section is devoted to the proof of Lemma~\ref{lemma:bangbang},
which shows that any feasible policy that does not have the above
bang-bang structure can be improved upon by making a perturbation
`towards this structure'. We start with the following lemma.
\begin{lemma}
  \label{lemma:monotoninity_rho}
  Assuming positive recurrence, for any $i \geq 1,$ $\pi(0)$ is a
  strictly decreasing function of $\rho_i$ and
  $\sum_{i=0}^{\infty} i \pi_i$ is a strictly increasing function of
  $\rho_i.$
\end{lemma}
The proof of Lemma~\ref{lemma:monotoninity_rho} is trivial; we omit
the proof.

\begin{proof}[Proof of Lemma~\ref{lemma:bangbang}]
  If $\bar{\rho} < 1$ and $\frac{\bar{\rho}}{1-\bar{\rho}} \leq C,$
  then the optimal solution, in light of
  Lemma~\ref{lemma:monotoninity_rho}, is: $\rho_i = \bar{\rho}$ for
  all $i$ (i.e., $\ell^* = \infty$).

  Thus, in the remainder of this proof, we assume that either (i)
  $\bar{\rho} \geq 1,$ or (ii) $\bar{\rho} < 1$ and
  $\frac{\bar{\rho}}{1-\bar{\rho}} > C.$ In this case, in light of
  Lemma~\ref{lemma:monotoninity_rho}, moment constraint
  in~\eqref{eq:bang_bang_obj} must hold with equality at the optimum.

  Let
  \begin{align*}
    f(\rho) &:= \frac{1}{\pi_0} = \sum_{j=0}^{\infty} h_j,\\
    m(\rho) &:= \sum_{i=0}^{\infty} (i-C) h_i. 
  \end{align*}
  Denoting the optimal solution of~\eqref{eq:bang_bang_obj} by
  $\rho^*,$ note that $m(\rho^*) = 0.$
  
  The statement of the lemma now follows from the following claim.

  {\bf Claim~1:} Consider $\rho = (\rho_i,\ i\geq 1)$ such that
  \begin{itemize}
  \item  $m(\rho) = 0,$ and
  \item there exists $i \geq 1$ satisfying $\rho_i < \bar{\rho},$
    $\rho_{i+1} > \underline{\rho}.$
  \end{itemize}
  Then $\rho$ is not optimal. Specifically, one can construct
  $\tilde{\rho}$, where $\tilde{\rho}_j = \rho_j$ for $j \ne i, i+1,$
  and for $\epsilon, \delta > 0,$
  $\tilde{\rho}_i = \rho_i + \epsilon,$
  $\tilde{\rho}_{i+1} = \rho_{i-1} - \delta,$ such that
  $m(\tilde{\rho}) = 0,$ $f(\tilde{\rho}) > f(\rho).$

  {\bf Proof of Claim~1:} First, consider the difference
  $f(\tilde{\rho}) - f(\rho).$
  \begin{eqnarray}
    f(\tilde{\rho}) - f(\rho) &=&  \sum_j \tilde{h}_j - \sum_j h_j = h_{i - 1}
   \left ( ( \tilde{\rho}_i - \rho_i) + \sum_{j \geq i+1}^{\infty} \left
   (\prod_{k = i}^j \tilde{\rho}_k - \prod_{k = i}^j  \rho_k \right ) \right )  \nonumber\\
     &=& h_{i - 1} \bigg( (\tilde{\rho_i} - \rho_i) + (\tilde{\rho_i}\tilde{\rho}_{i+1} - \rho_i \rho_{i+1})
         +  \sum_{j \geq i+2}^{\infty} (\tilde{\rho}_i\tilde{\rho}_{i+1} - \rho_i \rho_{i+1})  \prod_{k = i+2}^j  \rho_k \bigg) \nonumber \\
     &=& h_{i - 1}\bigg( (\tilde{\rho}_i - \rho_i) +  (\tilde{\rho}_i\tilde{\rho}_{i+1} - \rho_i \rho_{i+1})
       \bigg(1 + \sum_{j \geq i+2}^{\infty} \prod_{k = i+2}^j  \rho_k \bigg) \bigg). \label{eq:bangbang_1}
\end{eqnarray}

Next, consider the difference $m(\tilde{\rho}) - m(\rho).$
\begin{eqnarray}
\label{Eqn_diff_constraints}
m(\tilde{\rho}) - m(\rho)  &=& \sum_j \tilde{h}_j(j - C) \  - \ \sum_j h_j(j - C)= \sum_j (\tilde{h}_j - h_j) (j - C)\nonumber  \\
&=& h_{i-1}\bigg( (\tilde{\rho}_i - \rho_i)(i - C) +    (\tilde{\rho}_i\tilde{\rho}_{i+1} - \rho_i \rho_{i+1}) \bigg( (i+1-C) + \sum_{j \geq i+2} (j - C) \prod_{k = i+2}^{j} \rho_k    \bigg) \bigg). \nonumber
\end{eqnarray}
Setting the above difference to zero gives us the following condition
that relates the perturbations $\epsilon$ and $\delta.$
$$
\tilde{\rho}_i\tilde{\rho}_{i+1} - \rho_i \rho_{i+1} = \frac{-(\tilde{\rho}_i - \rho_i)(i - C)}{(i+1-C) + \sum_{j \geq i+2} (j - C) \prod_{k = i+2}^{j} \rho_k}  .
$$
Substituting the above into \eqref{eq:bangbang_1}, we have
\begin{eqnarray*}
	f(\tilde{\rho}) - f(\rho)   
	&=& h_{i - 1} \bigg( (\tilde{\rho_i} - \rho_i) - \frac{(\tilde{\rho}_i - \rho_i)(i - C)\big(1 + \sum_{j \geq i+2}^{\infty} \prod_{k = i+2}^j  \rho_k \big)}{(i+1-C) + \sum_{j \geq i+2} (j - C) \prod_{k = i+2}^{j} \rho_k}  \bigg) \\ 
	&&= \ \frac{h_{i - 1} (\tilde{\rho_i} - \rho_i)}{(i+1-C) + \sum_{j \geq i+2} (j - C) \prod_{k = i+2}^{j} \rho_k} \bigg( 1 + \sum_{j \geq i+2} (j - i) \prod_{k = i+2}^{j} \rho_k  \bigg ).
\end{eqnarray*}

It thus suffices to show that
\begin{equation}
  \label{eq:bangbang_2}
  (i+1-C) + \sum_{j \geq i+2} (j - C) \prod_{k = i+2}^{j} \rho_k > 0.
\end{equation}
This is trivial if $i+1 \geq C.$ Suppose then that $i+1 < C.$ In
this case, note that
$$m(\rho) = \sum_{j = 0} ^i (j-C) h_j + \sum_{j = i+1}^{\infty} (j-C) h_j = 0.$$
Since the first term above is negative, it follows that the second is
positive, i.e.,
$$\sum_{j = i+1}^{\infty} (j-C) h_j = h_{i+1} \left((i+1-C) + \sum_{j
    \geq i+2} (j - C) \prod_{k = i+2}^{j} \rho_k \right) >
0,$$
which implies \eqref{eq:bangbang_2}.
\end{proof}

\subsection{Proofs of results in Section~\ref{sec:universal_bounds_loss}}

\begin{proof}[Proof of Lemma~\ref{lemma:ub1}]
  The proof follows by considering only the revenue component of the
  objective, and via an application of Jensen's inequality.
  $$C_{\text{rel}} \leq \sum_{i = 0}^{\infty} \pi_i g^{-1}(\mu_i) \leq
  g^{-1}\left(\sum_{i = 0}^{\infty} \pi_i \mu_i \right) \leq
  g^{-1}(\lambda).$$ In the last step above, we use
  $\lambda = \sum_{i = 1}^{\infty} \pi_i \mu_i \leq \sum_{i =
    0}^{\infty} \pi_i \mu_i.$
\end{proof}

\begin{proof}[Proof of Lemma~\ref{lemma:ub2}]
  Consider a tagged server arriving into the system in steady
  state. If the server is matched at price~$p,$ its mean sojourn time
  is at least $\frac{1}{g(p)}.$ Optimizing with respect to~$p$ yields
  the upper bound. Since we are ignoring congestion from other
  (waiting) servers, this bound is expected to be tight in \emph{light
    traffic}.
\end{proof}

\subsection{Proof of Lemma~\ref{lem:ub}}\label{app:ub}
\begin{proof}
  Consider stationary Markov policies which apply a finite collection
  of prices $\{p_k,\ k \in \K\},$ such that the long run fraction of
  time price $p_k$ is used by the platform is $\theta_k > 0.$ Since the customer
  queue must be stable, we have
  $$\lambda \geq \sum_{k \in \K} \theta_k \mu(p_k).$$

Let $N_{k,t}$ denote the number of times price~$p_k$ has been
  used until time~$t,$ and $B_{k,i}$ is the random number of customer arrivals that take place  the $i$th time price $p_k$ is used. Clearly,
  $\{B_{k,i}\}_{i \geq 1}$ is an i.i.d. sequence with mean $\mu(p_k).$
  Using this, we bound the objective as follows.
  \begin{align*}
    \frac{1}{T} \sum_{t = 1}^T V(\sum_{i \in M(t)} p_i) &\leq V\left(\frac{\sum_{t = 1}^T \sum_{i \in M(t)} p_i}{T} \right) \\
    &=V\left(\frac{\sum_{t = 1}^T p(t) B(t) - \sum_{j \in C(t+1)} p_j  }{T} \right) \\
                                                        &=V\left(\frac{\sum_{k \in \K} p_k \sum_{i = 1}^ {N_{k,t}} B_{k,i}   - \sum_{j \in C(t+1)} p_j  }{T} \right)\\
                                                        &= V\left(\sum_{k \in \K} p_k \frac{\sum_{i = 1}^ {N_{k,t}} B_{k,i}}{N_{k,t}} \frac{N_{k,t}}{T} - \frac{\sum_{j \in C(t+1)} p_j}{T}     \right)
  \end{align*}
%
  Letting $T \ra \infty,$ we get that almost surely,
  \begin{align*}
    \limsup_{T \ra \infty} \frac{1}{T} \sum_{t = 1}^T V(\sum_{i \in M(t)} p_i) & \leq V\left(\sum_{k \in K} p_k \mu(p_k) \theta_k\right)\\
    & \leq V\left(p^\star \mu(p^\star) \right).
  \end{align*}
  The last step above uses the concavity of the function~$p \mu(p).$  
\end{proof}

\subsection{Proof of Theorem \ref{thm:lbq}}
\begin{proof}
With algorithm $\cA$  \eqref{eq:alg}, since a constant price $p(t) = p^\star$ is charged to all the customers, we rewrite 
the profit \eqref{eq:profit} as 
  \begin{equation}\label{eq:profitA}
  \bV_\cA= 
  \lim_{T\rightarrow \infty}\frac{1}{T}\sum_{t=1}^TV(p^\star
  |M_\cA(t)|).
\end{equation}
  Note that this rewriting of the profit \eqref{eq:profit} is possible only for algorithms that charge a constant price across time, however, $ \bV_\opt$ allows arbitrary pricing $p(t)$.
  
  Next, we show that
  $$\lim_{T\rightarrow \infty}\frac{1}{T}\sum_{t=1}^TV(p^\star
  |M_\cA(t)|) \ge V(p^\star \mu(p^\star)) - O\left(\left(\frac{\log
        U}{U}\right)^2\right).$$

Let $V^+$  and $V^{-}$ be the profit made by the algorithm $\cA$ when $N \geq U/2$   by using $|M_\cA(t)| = \mu(p^\star) + \delta$,  and when $N <  U/2$ (assuming $N(t) > \mu(p^\star) - \delta$), with $|M_\cA(t)| =\mu(p^\star) - \delta$, respectively.
The Taylor series expansion of the profit ($V^+$ and $V^-$) about $p^\star\mu(p^\star)$ can be written as 
\begin{align}
\begin{split}
\label{j_plus}
V^{+}= V(p^\star(\mu(p^\star) + \delta)) &= V\left(p^\star\mu(p^\star)\right) + V^{(1)}\left(p^\star\mu(p^\star)\right)\delta' + V^{(2)}\left(p^\star\mu(p^\star)\right)\delta'^{2} + o(\delta'^{2}),
\end{split} \\
\begin{split}
\label{j_minus}
V^{-}=V(p^\star(\mu(p^\star) - \delta)) &=  V\left(p^\star\mu(p^\star)\right) - V^{(1)}\left(p^\star\mu(p^\star)\right)\delta' + V^{(2)}\left(p^\star\mu(p^\star)\right)\delta'^{2} + o(\delta'^{2}).
\end{split}
\end{align}
where $V^{(i)}$ is the $i^{th}$ derivative of $V$ and $\delta'= p^\star \delta$.

Define $\rho^{+}$ as the fraction of time that $N > \frac{U}{2}$, and $\rho^{-}$ as the fraction of time that $N\le  \frac{U}{2}$. Then the profit \eqref{eq:profitA} can be written as
\begin{align}
\nonumber
\bV_\cA &\stackrel{(a)}= \rho^{+}V^{+} + \left( \rho^{-} - P_{\text{outage}} \right)V^{-}, \\
\label{upper-bound-large-B}
\begin{split}
&\stackrel{(b)}= V(p^\star \mu(p^\star)) + V^{(1)}(p^\star \mu(p^\star))(\rho^{+}\delta' - \rho^{-}\delta') + \Theta\left( \frac{(\log U)^2}{U^2}\right),
\end{split}
\end{align}
  where (a) follows from the fact that at most $|M_\cA(t)| = \mu(p^\star) - 
    \delta$ when $N <  U/2$ and $V^{-}$ is overestimated, (b) follows from \eqref{j_plus} and \eqref{j_minus} and Lemma \ref{lem:hittingprob}, i.e.,  $P_{\text{outage}} = \Theta\left(U^{-\beta}\right)$, with $\delta = \Theta\left(\frac{\log U}{U}\right)$,  $\rho^{+} + \rho^{-} = 1$, and $V^{(2)}(\mu(p^\star))$ is constant.

By applying the conservation of customer arrivals, i.e., the departed customers  is equal to the arrived customers, we have
 $\rho^{+} \left(\mu(p^\star) + \delta \right) + (\rho^{-} - P_{\text{outage}})  \left(\mu(p^\star) - \delta \right) 
          = \mu(p^\star)$, following \eqref{eq:alg}.

Rearranging this, and substituting the expression for $P_{\text{outage}}$ from Lemma \ref{lem:hittingprob}, we get  \begin{equation}
\label{conservation-of-packets}
\rho^{+}\delta - \rho^{-}\delta = \Theta\left(U^{-\beta}\right).
\end{equation}
Using \eqref{conservation-of-packets}, from \eqref{upper-bound-large-B}, 
\begin{align}
\nonumber
\bV_\cA &\stackrel{(a)} \geq V(p^\star \mu(p^\star)) + \Theta\left(U^{-\beta}\right) + \Theta\left(\frac{(\log U)^2}{U^2}\right). 
\end{align}
From Lemma \ref{lem:ub}, $\bV_\opt \le V(p^\star \mu(p^\star))$. Hence, we get that 
$$\bV_\opt - V_\cA \le \Theta\left(\frac{(\log U)^2}{U^2}\right), $$
since $\beta \ge 2$. 
\end{proof}

Next, provide the remaining proof of Lemma \ref{lem:hittingprob}.
\begin{proof}[Proof of Lemma \ref{lem:hittingprob}]
Let the event $N <  \mu(p^\star) - \delta$ be called as the {\bf outage} event. Since we are letting $U$ large, the outage event is similar to $N=0$. Hence, we will try to bound the outage probability $\bbP(N =0) = \lim_{T \rightarrow \infty}\frac{1}{T} \sum_{t=1}^T \bbP(N(t) =0)$.

We will break the time into intervals $\cI$, where each interval has $\frac{U}{2L}>1$ time slots, and $L$ is some constant to be chosen later. Without loss of generality, we let $\frac{U}{2L}$ to be an integer. We call slot $t \in \cI_i$ if slot $t$ falls in  the interval $\cI_i = [i\frac{U}{2L}, \ \ (i+1)\frac{U}{2L}]$. Let the system be in operation since time $-\infty$.

Then {\it event} $E_i$ is defined as $N(0)=0$ and $\cI_{-i}$ be the last interval during which $N(t) \ge U/2$, i.e., 
$t \in \cI_{-i}$.

The basic idea in defining $E_i$, is that throughout time consisting of $i$ intervals, starting from interval
$\cI_{-i}$ till interval $\cI_0$, the algorithm \eqref{eq:alg} will be using $|M_\cA(t)| = \mu(p^\star) - \delta$ since $N(t)\ge U/2$ throughout, while the arrival process $N(t)$ has increments with mean at least $\mu(p^\star)$. Thus, there is a positive bias to the process $N(t)$,  
during  interval
$\cI_{-i}$ till interval $\cI_0$, and hence the probability $P(E_i)$ is expected to be exponentially small with respect to $U$. 

Under these definitions, 
$$\bbP(N =0) = \sum_{i=0}^\infty P(E_i).$$
 We define two
events that only depend on server arrivals $A(t)$ and customer arrivals $B(t)$ in time slot $t$, respectively.  
 Let 
$$F_{i1}: \sum_{t = -i\frac{U}{2L}}^0 \left(|M_\cA(t)| - B(t) \right) > U/2,$$
and
$$F_{i2} : \sum_{t = -i\frac{U}{2L}}^0 \left(|M_\cA(t)| - A(t) \right) >
U/2.$$

We claim that event $E_i$ implies that at least one of $F_{i1}$ or $F_{i2}$, since the arrivals for either $A(t)$ or $B(t)$ are insufficient
compared to $|M_\cA(t)|$ for event $F_{i1}$ or $F_{i2}$ to happen. 

To prove the claim consider the following two cases. 
The condition we know is that at time $t$, $N(t) =  \min\{S(t), C(t)\} = U/2$. 
Case I : $S(t) = U/2$ and $C(t) \ge U/2$. 
Therefore, if both $F_{i1}$ and $F_{i2}$ are false, then $N(0)>0$, and hence $E_{i}$ is also false. 
Identical argument holds when 
Case II : $S(t) \ge U/2$ and $C(t) = U/2$. 

Thus, we have that $E_i $ implies $F_{i1} \cup F_{i2}$, and therefore $P(E_i) \le P(F_{i1}) + P(F_{i2})$. 

Moreover, since $B(t)$
(with arrival rate $\lambda > \mu(p^\star))$ stochastically
dominates $A(t)$ (with arrival rate $\mu(p^\star)$), we have
$P(F_{i1})\le P(F_{i2})$. Hence, we have $$P(E_i) \le 2P(F_{i2}).$$
B(t)

Next, we upper bound $P(F_{i2})$ using 
Chernoff's bound as follows, where the proof is similar to \cite{srivastava2012basic}, and is provided for completeness..

We begin by noting that $P(F_{i1}) = P\left(\sum_{t = -i\frac{U}{2L}}^0 \left(|M_\cA(t)|-B(t) \right) > U/2\right)$ and 

$ P\left(\sum_{t = -i\frac{U}{2L}}^0 \left(|M_\cA(t)|-B(t) \right) > U/2\right)$
\begin{align}
 \stackrel{(a)}\le&  
  \bbE\left\{\exp\left( \theta_i \sum_{t = -i\frac{U}{2L}}^0 \left(|M_\cA(t)|-B(t) \right) \right) \right\} \exp\left(-\theta_i U/2\right)\\
  \stackrel{(b)}\le& \bbE\left\{\exp\left( - \theta_i \sum_{t = -i\frac{U}{2L}}^0  B(t) \right) \right\} 
  \exp\left(\theta_i \frac{U}{2L} \left(\mu(p^\star) - \delta\right) \right)\exp\left(-\theta_i U/2\right), \\ \label{eq:ub1}
  \stackrel{(b)}=&      \exp \left( -\frac{U}{2} \left(\theta_i \left(1- \frac{i+1}{L}(\mu(p^\star) - \delta)\right) - \frac{i+1}{L}\cM_C(-\theta_i) + \epsilon_i(U,\theta_i)\right)\right),
  \end{align}
where $(a)$ follows from Chernoff's bound for $\theta_i\ge 0$, $(b)$ follows since $|M_\cA(t)| \le \mu(p^\star) - \delta$ for all slots $t $ from $-i\frac{U}{2L}$ till $0$, where for $(c)$ we define $\cM_C(s) = \lim_{T\rightarrow \infty}\frac{1}{T}\bbE\left\{\exp\left(s \sum_{t=1}^T B(t)\right)\right\}$ as the semi-invariant log-moment generating function of $B(t)$ the number of customer arrivals in slot $t$, and $\epsilon_i(U,\theta_i) \rightarrow 0$ as $U \rightarrow \infty$. Note that $\epsilon_i(U,\theta_i)$ appears since the limit of the summation in $(a)$ is from $t = -i\frac{U}{2L}$ to $0$, and not from $t=0$ to $\infty$ and no limits are taken.

From \eqref{eq:ub1}, let $f_i(\theta) = \left(\theta \left(1- \frac{i+1}{L}(\mu(p^\star) - \delta)\right) - \frac{i+1}{L}\cM_C(-\theta)\right),$
and 
$$
\theta_i^* = \arg \max_{\theta} f_i(\theta).$$

Note that $\bbE\{B(t)\} = \mu(p^\star) > \mu(p^\star)  - \delta$. Hence the function $\theta(\mu(p^\star)- \delta) + \cM_C(-\theta)$ has a negative slope at $\theta=0$. Hence, there exists a $\theta'$ such that $\theta'(\mu(p^\star)- \delta) + \cM_C(\theta') < 0$. Moreover, we get that for such a $\theta'$, there exists $\sfi$ such that for $i \ge \sfi $ and $\gamma>0$, such that 
\begin{equation}\label{eq:dum1}
f_i(\theta') \ge f_\opt + \sfi \gamma,
\end{equation}
where $$f_\opt = \inf_{i\ge 0} \sup_{\theta\ge 0} f_i(\theta).$$
Recall that 
\begin{equation}\label{eq:dum2}
\bbP(N =0) = \sum_{i=0}^\infty P(E_i) \le 2 \sum_{i=0}^\infty P(F_{i2}).
\end{equation}
Writing a partial sum $P(F_{i2})$ for $i\ge \sfi$ from \eqref{eq:ub1}, 
\begin{align}\nn\sum_{i\ge \sfi}P(F_{i2}) & \le \sum_{i\ge \sfi}\exp\left(-\frac{U}{2} \left(\theta_i \left(1- \frac{i+1}{L}(\mu(p^\star) - \delta)\right) - \frac{i+1}{L}\cM_C(\theta_i)+ \epsilon_i(U,\theta_i)\right)\right), \\
& \le \sum_{i\ge \sfi}\exp\left(-\frac{U}{2} \left(f_i(\theta')+ \epsilon_i(U,\theta')\right)\right), \\
& \stackrel{(a)}\le \sum_{i\ge \sfi}\exp\left(-\frac{U}{2} \left(f_\opt + \sfi \gamma + \inf_{i\ge \sfi}\epsilon_i(U,\theta')\right)\right), \\
\label{eq:dummy21} 
& = \exp\left(-f_\opt \frac{U}{2L}\right)
\frac{\exp\left(-\frac{U}{2} \left((\sfi+1)\gamma + \inf_{i>\sfi}\epsilon_i(U,\theta')\right)\right)}{1-\exp\left(-\gamma\frac{U}{2}\right)},
\end{align}
where $(a)$ follows from  \eqref{eq:dum1}. 
Thus, $\sup_{U \rightarrow \infty}\frac{2}{U}\log (\sum_{i\ge \sfi}P(F_{i1})) = 0$.

The sum of $P(F_{i2})$ for the first $\sfi-1$ terms is at most
\begin{align}\nn\sum_{i< \sfi}P(F_{i2})& \le \sum_{i < \sfi}\exp\left(-\frac{U}{2} \left(\theta_i \left(1- \frac{i+1}{L}(\mu(p^\star) - \delta)\right) - \frac{i+1}{L}\cM_C(\theta_i)+ \epsilon_i(U,\theta_i)\right)\right), \\
\label{eq:dummy22}
& \le \sum_{i\le \sfi}\exp\left(-f_\opt \frac{U}{2}\right)\left((\sfi+1)\min_{0\le i < \sfi}\exp( \epsilon_i(U,\theta_i^*)\right).
\end{align}
Combining, \eqref{eq:dummy21} and  \eqref{eq:dummy22}, $$\sup_{U \rightarrow \infty}\frac{2}{U}\log (\sum_{i=0}^\infty P(F_{i2})) = -f_\opt.$$

Thus, from \eqref{eq:dum2}
\begin{equation}\label{eq:dum3}
\sup_{U \rightarrow \infty}\frac{2}{U}\log\bbP(N =0) \le -f_\opt.
\end{equation}
This is true for all values of $L$, thus we let $L \rightarrow \infty$, and tighten the bound \eqref{eq:dum3} as follows by using the definition of $f_\opt$.
\begin{align}\label{}
  \sup_{U \rightarrow \infty}\frac{2}{U}\log\bbP(N =0) \le -f_\opt=&-  \inf_{i\ge 0} \sup_{\theta\ge 0} f_i(\theta), \\
  =& \inf_{i\ge 0} \sup_{\theta\ge 0}  \left(\theta \left(1- \frac{i}{L}(\mu(p^\star) - \delta)\right) - \frac{i}{L}\cM_C(-\theta)\right), \\
  =&  \inf_{T\ge 0} \sup_{\theta\ge 0}  \left(\theta \left(1- T(\mu(p^\star) - \delta)\right) - T\cM_C(-\theta)\right), \\
=&  \inf_{T\ge 0} \sup_{\theta\ge 0}  T \left(-\theta \left(-\frac{1}{T} + (\mu(p^\star) - \delta)\right) - \cM_C(-\theta)\right).
\end{align}
This infimum and supremum is achieved at $T=\tau^*$ and $\theta = T$, where $\tau^*$ is such that 
$\cM_C(-\tau^*) = \tau^*(\mu(p^\star)-\delta)$.

Hence, we get that $$
  \sup_{U \rightarrow \infty}\frac{2}{U}\log\bbP(N =0) \le - \tau^*.$$
  
 Rewriting $\tau^*$ as the root of $g(\tau) = \cM_C(-\tau^*) - \tau^*(\mu(p^\star)-\delta)$, 
Lemma \ref{lem:exp} implies that $\frac{d\tau^*}{d\delta} = -\frac{2}{\sigma^2_C} + o(\delta)$. Thus, we get $\tau^* = \frac{2 \delta}{\sigma^2_C} + o(\delta)$.

Therefore, choosing $\delta = \alpha \frac{\log U}{U}$ and $\alpha = \beta \sigma^2_C$ for $\beta\ge 2$, we get 
$$\bbP(N =0) = O(U^{-\beta}).$$

\end{proof}
 \begin{lemma}\label{lem:exp}
 $$ \frac{d\tau^*}{d\delta}\Bigr|_{\delta=0} = -\frac{2}{\sigma^2_C},$$
 where $\sigma^2_C = \lim_{T\rightarrow \infty} \frac{1}{T} \text{var}\left(\sum_{t=1}^T B(t)\right)$, the variance of process $B(t)$, the number of arrivals of customers.
\end{lemma}
Recall that since the price $p(t)$ is a constant, process $B(t)$ is an i.i.d. process and $\sigma^2_C$ is well defined.

\end{document}